\documentclass[11pt]{article}
 \usepackage{amsmath,amssymb,amsbsy,amsfonts,amsthm,latexsym,amsopn,amstext, amsxtra,euscript,amscd}
\begin{document}

%  use the AMS-Euler Fraktur fonts
%%%%%%%%%%%%%%%%%%%%%%%%%%%%%%%%%%
\newfont{\teneufm}{eufm10}
\newfont{\seveneufm}{eufm7}
\newfont{\fiveeufm}{eufm5}
%%%%%%%%%%%%%%%%%%%%%%%%%%%%%%%%%
%
%  allow automatic size selection in math mode
%
%%%%%%%%%%%%%%%%%%%%%%%%%%%%%%%%%
\newfam\eufmfam
                    \textfont\eufmfam=\teneufm \scriptfont\eufmfam=\seveneufm
                    \scriptscriptfont\eufmfam=\fiveeufm
%%%%%%%%%%%%%%%%%%%%%%%%%%%%%%%%%

%
%  \frak works on a single symbol at a time...
%
\def\frak#1{{\fam\eufmfam\relax#1}}
%

%%%%%%%%%%%%%%%%%%%  bbb-matter

\newtheorem{thm}{Theorem}
\newtheorem{lemma}[thm]{Lemma}
\newtheorem{claim}[thm]{Claim}
\newtheorem{cor}[thm]{Corollary}
\newtheorem{prop}[thm]{Proposition}
\newtheorem{definition}{Definition}
\newtheorem{question}[thm]{Open Question}
\newtheorem{remark}[thm]{Remark}
\newtheorem{conjecture}[thm]{Conjecture}

\def\squareforqed{\hbox{\rlap{$\sqcap$}$\sqcup$}}
\def\qed{\ifmmode\squareforqed\else{\unskip\nobreak\hfil
\penalty50\hskip1em\null\nobreak\hfil\squareforqed
\parfillskip=0pt\finalhyphendemerits=0\endgraf}\fi}%%

\hyphenation{re-pub-lished}

\def\dist{\mathrm{dist}}
\def\btau{\overline\tau}
\def\ord{{\mathrm{ord}}}
\def\Nm{{\mathrm{Nm}}}
\def\Sp{{\mathrm{Sp}}}
\renewcommand{\vec}[1]{\mathbf{#1}}

\def \C{{\mathbb{C}}}
\def \F {{\mathbb{F}}}
\def \L{{\mathbb{L}}}
\def \K{{\mathbb{K}}}
\def \Z{{\mathbb{Z}}}
\def \N{{\mathbb{N}}}
\def \Q{{\mathbb{Q}}}
\def \R{{\mathbb{R}}}

\newcommand{\Fq}{{\mathbb F}_q}
\newcommand{\Fqx}{\Fq^\star}
\newcommand{\Fqr}{{\mathbb F}_{q^r}}
\newcommand{\Fqm}{{\mathbb F}_{q^m}}

\def\mand{\qquad \mbox{and} \qquad}

%%%%%%%%%%%%%%%  Topmatter %%%%%%%%%%%%%%%%%%

\title{\Large Decomposing Jacobians of Curves over Finite Fields in the Absence of Algebraic Structure}

\iffalse
\author{
\begin{tabular}{c}
Omran Ahmadi\footnote{Research supported by IPM}.
\\[2mm]School of Mathematics\\ Institute for Research in Fundamental Sciences (IPM), Tehran\\ Iran
\\[2mm]
Gary McGuire\footnote{ Research supported by the Claude Shannon 
Institute, Science
Foundation Ireland Grant 06/MI/006, email:gary.mcguire@ucd.ie } \\[2mm]
School of Mathematical Sciences\\
University College Dublin\\
Ireland
\end{tabular}}
\fi
%\baselineskip=16.3pt
%\parskip=14pt
\author{
 {\sc Omran Ahmadi}  \\
{School of Mathematics}\\{ Institute for Research in Fundamental Sciences (IPM)}\\
{ Tehran, Iran}\\
{\tt oahmadid@ipm.ir}
\and
{\sc Gary McGuire} \\
{School of Mathematical Sciences}\\ {University College Dublin} \\
{Ireland} \\
{\tt gary.mcguire@ucd.ie}
\and
{\sc Antonio Rojas-Le\'on}  \thanks{Partially supported by MTM2010-19298 (Min. Ciencia e Innovaci\'on) and FEDER} \\
{Department of Algebra}\\ {University of Seville} \\
{Spain} \\
{\tt arojas@us.es}
}
\date{}

%\begin{document}
\maketitle

\subsection*{Abstract}
We consider the issue of when the L-polynomial of one curve  over $\F_q$ divides
 the L-polynomial of another curve.
 We prove  a theorem which shows that divisibility follows
 from a hypothesis that two curves have the same 
 number of points over infinitely many extensions of a certain type,
 and one other assumption.
 We also present an application to a family of curves arising
 from a conjecture about exponential sums.
 We make our own conjecture about L-polynomials, and prove that this is equivalent
 to the exponential sums conjecture.
 \bigskip
 
Keywords: curve, Jacobian, supersingular, finite field, L-polynomial.

MSC: 14H45, 11M38

%\newpage

\section{Introduction}

Let $q=p^a$ where $p$ is a prime, and let $\F_q$ denote the finite field with $q$ elements.
Let  $C=C(\F_q)$ be a projective smooth absolutely irreducible
curve of genus $g$ defined over $\F_q$.
For any $n\geq 1$ let  $C(\F_{q^n})=C(\F_q) \otimes_{\F_q} \F_{q^n}$ be the set of
$\F_{q^n}$-rational points of $C$, and let
 $\#C(\F_{q^n})$ be the cardinality of $C(\F_{q^n})$.
 Similarly, if $\overline{\F_q}$ denotes a fixed algebraic closure of $\F_q$,
let $C(\overline{\F_q})=C(\F_q) \otimes_{\F_q} \overline{\F_q}$.

The divisor group of $C$ 
is the free abelian group generated by the points of $C(\overline{\F_q})$. 
Thus, a divisor is a formal sum $\sum n_P P$ over
all $P\in C(\overline{\F_q})$, where all but finitely many $n_P$ are 0.
The degree of a divisor is $\sum n_P $.
The divisor of a function in the function field $\overline{\F_q}(C)$
must have degree 0, and is called a principal divisor.
The quotient of the subgroup of degree 0 divisors by the principal divisors
is denoted $Pic^0 (C(\overline{\F_q}))$, and is canonically isomorphic to
the Jacobian of $C$, $Jac(C) (\overline{\F_q})$, after a point at infinity is chosen.
The Galois group $Gal(\overline{\F_q}/\F_q)$ acts on 
divisors and divisor classes, and we define
$Jac(C)=Jac(C)(\F_q)=Pic^0(C)=Pic^0 (C(\F_q))$ to be the divisor classes that are fixed
by every element of $Gal(\overline{\F_q}/\F_q)$.
The Jacobian $Jac(C)$
 is an abelian variety of dimension $g$ defined over $\F_q$.

The Frobenius map $\pi: 
x\mapsto x^q$ on $\overline{\F_q}$ induces a Frobenius map
 on $C(\overline{\F_q})$.
The elements of $C(\F_{q^n})$ are the fixed points of $\pi^n$.
The Frobenius morphism $\pi$ induces a map on divisor classes, 
and hence on the Jacobian,
 and hence a Frobenius endomorphism
 on the $\ell$-adic Tate module $V_\ell  (Jac(C))$.
Let $P_C(t)$ denote the characteristic polynomial
of the Frobenius endomorphism, which is known to have integer coefficients.
An abelian variety defined over $\F_q$ is called $\F_q$-simple if 
it is not isogenous over $\F_q$ to a product of abelian varieties of lower dimensions. 
An abelian variety is absolutely simple if it is $\overline{\F_q}$-simple.
If $Jac(C)$ is $\F_q$-simple 
then it can be shown that $P_C(X)=h(X)^e$ where $h(X)\in \mathbb{Z}[X]$ is 
irreducible over $\mathbb{Z}$ and $e\geq 1$.
We refer the reader to Waterhouse \cite{waterhouse} for 
these and further details about abelian varieties.

Given an abelian variety $A$ of dimension $g$ defined over $\F_q$, 
for a prime $\ell \not= p$ one defines $A[\ell]$ as the group of points
on $A$ (with values in an algebraic closure $\overline{k}$)
of order dividing $\ell$. Like in the classical case over $\mathbb{C}$
it can be shown that $A[\ell]$ is a $2g$-dimensional $\Z/\ell\Z$-vector
space. Things are different when $\ell =p$.
The \emph{$p$-rank} of $A$ is defined by
\begin{displaymath}
                         r_p (A) = \dim_{\mathbb{F}_p} A[p](\overline{k}),
\end{displaymath}
where $A[p](\overline{k})$ is the subgroup of $p$-torsion points over the algebraic closure. 
The $p$-rank can take any value between $0$ and  $g=\dim(A)$. 
When $r_p(A)=g$ we say that $A$ is ordinary.
The number $r_p(A)$ is invariant under isogenies over $k$, and satisfies $r_p (A_{1} \times A_{2} ) = r_p (A_{1}) + r_p (A_{2} )$.

The zeta function of $C$ is defined by 
\[
Z_C(t ) = exp \biggl( \sum_{n\geq 1} \#C(\F_{q^n})  \frac{t^n}{n} \biggr)
= exp \biggl( \sum_{n\geq 1} \# Fix(\pi^n)  \frac{t^n}{n} \biggr).
\]
It was shown by Artin and Schmidt (see Roquette \cite{rv}) 
that $Z_C(t)$ can be written in the form
  \[
   \frac{L_C(t)}{(1-t)(1-qt)}
   \]
  where $L_C(t) \in \mathbb{Z}[t]$ (called the L-Polynomial of $C$) is of degree $2g$.
  Weil showed that  $L_C(t)=t^{2g} P_C(1/t) $, and therefore 
factorizations of $P_C(t)$ are equivalent to factorizations of $L_C(t)$.

The characteristic polynomial of Frobenius
carries a lot of information about an abelian variety.
In fact, the isogeny classes of abelian varieties
are completely classified by their characteristic polynomials,
as the following theorem of Tate shows.

\begin{thm}(Tate) \label{Tate} 
Let $A$ and $B$ be abelian varieties defined over $\mathbb{F}_q$. Then an
abelian variety $A$ is $\mathbb{F}_q$-isogenous to an abelian subvariety of
$B$ if and only if $P_A(t)$ divides $P_B(t)$ over $\mathbb{Q}[t]$. In
particular, $P_A(t) = P_B(t)$ if and only if $A$ and $B$ are
$\mathbb{F}_q$-isogenous.
\end{thm}

When $Jac(C)$ is not $\F_q$-simple it decomposes up to isogeny 
(by Poincare's theorem) into a product
of abelian varieties of smaller dimensions, and Tate's theorem 
shows that the characteristic polynomial
$P_C(t)$ is divisible by the characteristic polynomials of the subvarieties.
In this paper we are interested in this phenomenon of the
decomposition of $Jac(C)$.

It follows from this discussion that decomposing the Jacobian up to isogeny,
factorizing the characteristic polynomial, and factorizing the L-polynomial,
are all equivalent.

The decomposition of $Jac(C)$ has been studied in many papers before now,
see Aubry-Perret \cite{AubryPerret}, Paulhus \cite{Paulhus_decomposingjacobians} 
or Bauer-Teske-Weng \cite{journals/moc/BauerTW05}  for example,
and there are two well-known approaches.
The first  approach is to use the Kani-Rosen decomposition theorem \cite{KaniRosen},
applicable for many groups $G=Aut(C)$.
%The Kani-Rosen theorem concerns isogenies and idempotents in the group algebra $\mathbb{Q}[G]$.  

\begin{thm}\label{kr}  (Kani-Rosen)
Given a curve $C$, let $G\leq Aut(C)$ be a
finite group such that 
$G=H_1 \cup \cdots \cup H_t$ where
the $H_i$ are subgroups of $G$ such that
$H_i \cap H_j = \{ 1 \}$ if $i\not= j$.
Then we have the following isogeny relation
\[
 {\rm Jac}(C)^{t-1} \times {\rm Jac}({C/G})^{|G|} \sim {\rm Jac}({C/H_1})^{|H_1|}
 \times \cdots \times {\rm Jac}({C/H_t})^{|H_t|}.
 \]
\end{thm}
This usually yields a decomposition of the Jacobian and a factorization of $L_C(t)$.
For example, in the special case where $G$ is
the Klein 4-group with subgroups $H_1$, $H_2$, $H_3$, 
 the Kani-Rosen theorem implies an isogeny
 \[
 {\rm Jac}(C)^2 \times {\rm Jac}({C/G})^4 \sim {\rm Jac}({C/H_1})^2
 \times {\rm Jac}({C/H_2})^2\times {\rm Jac}({C/H_3})^2
 \]
 which implies the following L-polynomial relation
 \[
 L_C(t)\ L_{C/G} (t)^2=L_{C/H_1} (t)\ L_{C/H_2} (t)\ L_{C/H_3} (t).
 \]
 An example of a paper applying the theorem in the
case of the Klein-4-group is \cite{MR2648557}.
However, as pointed out in \cite{Paulhus_decomposingjacobians}, 
the Kani-Rosen theorem will not apply when $Aut(C)$ is cyclic.
Also, most curves have trivial automorphism group, so the theorem does not
apply to them. 

Related to this method is the fact that
the L-polynomial of a fibre product $C\times D$ is  divisible by the
L-polynomials of $C$ and $D$.

The second well-known approach is to use a theorem of Kleiman \cite{Kleiman},
also sometimes attributed to Serre,
which implies a decomposition of $Jac(C)$ whenever there
is a covering map $C\longrightarrow C'$.

%is the contravariant Picard functor (pullback map Jac(C_2) -> Jac(C) ) injective?

\begin{thm}\label{KleimanSerre} 
(Kleiman-Serre)
If there is a morphism of curves $C \longrightarrow C'$ that is defined over $\F_q$
then $L_{C'}(t)$ divides $L_C(t)$.
\end{thm}

%Similar theorems are discussed in \cite{AubryPerret}.

These two approaches show that $L_C(t)$ is divisible by the
L-polynomial of a curve that is a quotient of $C$ or a morphic image of $C$.
In this article we will add a third approach, which may apply to curves
 with no nontrivial automorphisms,
and in situations where there are no nontrivial covering maps.
%These are situations where the Kani-Rosen and Kleiman-Serre methods no longer apply.
We replace the hypotheses of algebraic structure with a hypothesis about
the number of rational points, and show that the Jacobians of
such curves can exhibit similar decomposition behaviour.
Specifically, we will prove the following theorem.

\begin{thm}\label{main}
Let $C$ and $D$ be two smooth projective curves over $\F_q$.
Assume there exists a positive integer $k$ such that \begin{enumerate}
\item $\#C(\F_{q^m})=\#D(\F_{q^m})$ for every $m$ that is not divisible by $k$, and
\item the $k$-th powers of the roots of $L_C(t)$ are all distinct. 
\end{enumerate}
Then  $L_D(t)=q(t^k)L_C(t)$ for some polynomial $q(t)$ in $\Z[t]$. 
\end{thm}

%We provide  examples to illustrate our results.

The theorem of Tate (Theorem \ref{Tate}) when applied to two elliptic curves $E_1$ and $E_2$ defined over $\F_q$ says that when $E_1$ and $E_2$ have the same number of $\F_q$-rational points, there must be an isogeny between the curves. 
Thus, two curves having the same number of points cannot be a combinatorial accident;
it must happen because of an isogeny.
Theorem \ref{main} may be seen as a generalization of this result when curves $C_1$ and $C_2$ are of different genera. 
Theorem \ref{main} says that if the two curves have the same number of points over
all the prescribed extension fields, then this is not a combinatorial 
accident but is coming from a geometric relationship between their Jacobians.
There may not be a relationship between the curves, such as a morphism, but
there must be a relationship between their Jacobians.

In Section 2 we will recall some background that we need for the article.
In Section 3 we will prove Theorem \ref{main}.
In Section 4 we will apply our results to a family of curves, and state our own
conjectures about this family.
In Section 5 we give our motivation for this work, which was a conjecture
on exponential sums.
Finally in Section 6 we prove equivalence of the conjectures.

\section{Background}

\subsection{L-polynomials}\label{lprelim}

It is traditional to write
\[
L_C(t)=\prod_{i=1}^{2g} (1-\alpha_i t), \ \textrm{and} \ P_C(t)=\prod_{i=1}^{2g} (t-\alpha_i ).
\]
The $\alpha_i$ are called the Frobenius eigenvalues of $C$ (or of $Jac(C)$),
because they are the eigenvalues of the $\F_q$-Frobenius endomorphism
action on the  $\ell$-adic Tate module $V_\ell  (Jac(C))$.
We briefly recall some further well-known facts about L-polynomials (see 
\cite{DBLP:books/daglib/0084861} for example).
 If $L^{(n)}(t)$ denotes the L-polynomial of $C(\F_{q^n})$ 
then $$L^{(n)} (t)=\prod_{i=1}^{2g} (1-\alpha_i^n t).$$
The number of rational points for all $n\geq 1$ is given by
$$\# C(\F_{q^n})=q^n+1-\sum_{i=1}^{2g} \alpha_i^n.$$
This means that the coefficient of $t$ in $L^{(n)} (t)$
is equal to $\# C(\F_{q^n})-(q^n+1)$.

 The algebraic integers $\alpha_i$ can be labelled so that 
$\overline{\alpha_i}=\alpha_{i+g}$  and $\alpha_i \alpha_{i+g}=q$, so
$|\alpha_i|=\sqrt{q}$.

%It follows from these  properties that the characteristic polynomial $P_C(t)$ 
%of the Frobenius endomorphism on $Jac(C)$� has the form
%\begin{displaymath}
%P_C(t) = t^{2g} + a_{1}t^{2g-1} + \cdots+ a_{g-1}t^{g+1}+
%a_gt^g +qa_{g-1}t^{g-1}+ \cdots + a_1 q^{g-1} t+q^g.
%\end{displaymath}

\subsection{Morphisms}\label{CurveMorph}
In this section first we recall some theorems which will be used later.

\begin{thm}\cite[Theorem 2.3]{Silverman}\label{non-constant morphism}
Let $\phi: C_1\longrightarrow C_2$ be a morphism of curves. Then $\phi$ is either constant or surjective.
\end{thm}

Now let $K$ be a field of characteristic $p>0$ and let $q=p^r$. If $C$ is a curve defined by a single equation $f=0$ over $K$
, then we can define a new curve $C^{(q)}$ which is the zero set of the equation $f^{(q)}=0$ where  $f^{(q)}$ is the polynomial obtained from $f$ by raising its coefficients to the power $q$. It follows that there is a natural morphism between $C$ and $C^{(q)}$ called the Frobenius morphism.

\begin{thm}\cite[Corollary 2.12]{Silverman}\label{factor-through-frob}
Every map $\psi:C_1\longrightarrow C_2$ of smooth curves over a field of characteristic $p>0$ factors as
$$
C_1\overset{\phi}{\longrightarrow} C_1^{(q)}\overset{\lambda}{\longrightarrow} C_2
$$
for some $q$ where the map $\phi$ is the  Frobenius map, and the map $\lambda$ is separable. 
\end{thm}
The following is an immediate corollary of the above theorem as with the assumptions of the following corollary, $C_1^{(q)}=C_1$.

\begin{cor}\label{Separable}
Let $p$ be a prime number, $\mathbb{F}_p$ the finite field with $p$ elements, and let $C_1, C_2$ be smooth curves defined over $\mathbb{F}_p$.  Furthermore suppose that there is a map $\psi:C_1\longrightarrow C_2$. Then there is a map $\lambda:C_1\longrightarrow C_2$ which is separable. 
\end{cor}

\section{Divisibility of L-polynomials}

In this section we prove our result on the divisibility relation between L-polynomials.
Again $q=p^a$ where $p$ is a prime.
After proving the theorem we will then prove a partial converse.

\subsection{Divisibility Theorem}

We begin with one preliminary lemma.

\begin{lemma}\label{radical-subrings}
Suppose that $f(x),g_1(x)$ and $g_2(x)$ are polynomials in $\Z[x]$ where $f(0)\neq 0$ and $f(x)^n=g_1(x^k)/g_2(x^k)$ for positive integers $k$ and $n$. Then there exists a polynomial $h(x)$ in $\Z[x]$ so that $f(x)=h(x^k)$. 
\end{lemma}

\begin{proof}
Let $\zeta_k$ be a primitive $k$-th root of unity. Then 
$$
f(\zeta_kx)^n=g_1((\zeta_kx)^k)/g_2((\zeta_kx)^k)=g_1(x^k)/g_2(x^k)=f(x)^n.
$$
Thus $f(\zeta_kx)=\zeta_n f(x)$ for some $n$-th root of unity $\zeta_n$. Since $f(0)\neq 0$ we have $\zeta_n=1$ and $f(\zeta_kx)=f(x)$. Now the claim follows.
\end{proof}

We restate Theorem \ref{main} and give the proof.

\begin{thm}
Let $C$ and $D$ be two smooth projective curves over $\F_q$.
Assume there exists a positive integer $k$ such that 
\begin{enumerate}
\item $\#C(\F_{q^m})=\#D(\F_{q^m})$ for every $m$ that is not divisible by $k$, and
\item the $k$-th powers of the roots of $L_C(t)$ are all distinct. 
\end{enumerate}
Then  $L_D(t)=q(t^k)L_C(t)$ for some polynomial $q(t)$ in $\Z[t]$. 
\end{thm}

\begin{proof}
Let $L_C(t)=\prod_{i=1}^{2g(C)}(1-\alpha_it)$ and $L_D(t)=\prod_{j=1}^{2g(d)}(1-\beta_jt)$, then 
$$
\#C(\Fqm)=1+q^m-\sum_{i=1}^{2g(C)}\alpha_i^m
$$
and
$$
\#D(\Fqm)=1+q^m-\sum_{j=1}^{2g(D)}\beta_j^m.
$$
So, by hypothesis, there exists a positive integer $k$ such that
$$
\sum_{i=1}^{2g(C)}\alpha_i^m=\sum_{j=1}^{2g(D)}\beta_j^m
$$
for every $m$ with $k\nmid m$. This gives an equality of certain 
zeta functions, namely
$$
\exp\left(\sum_{m:k\nmid m}\sum_{i=1}^{2g(C)}\alpha_i^m\frac{t^m}{m}\right)=\exp\left(\sum_{m:k\nmid m}\sum_{j=1}^{2g(D)}\beta_j^m\frac{t^m}{m}\right)
$$

Now
$$
\exp\left(\sum_{m:k\nmid m}\sum_{i=1}^{2g(C)}\alpha_i^m\frac{t^m}{m}\right)=\exp\left(\sum_{m}\sum_{i=1}^{2g(C)}\alpha_i^m\frac{t^m}{m}-\sum_{m:k | m}\sum_{i=1}^{2g(C)}\alpha_i^m\frac{t^m}{m}\right)=
$$
$$
=\exp\left(\sum_{m}\sum_{i=1}^{2g(C)}\alpha_i^m\frac{t^m}{m}-\sum_{m}\sum_{i=1}^{2g(C)}\alpha_i^{km}\frac{t^{km}}{km}\right)=\frac{\prod_{i=1}^{2g(C)}(1-\alpha_i^kt^k)^{1/k}}{\prod_{i=1}^{2g(C)}(1-\alpha_it)}=
$$
$$
=\frac{L_C^{(k)}(t^k)^{1/k}}{L_C(t)}
$$
Therefore,
$$
\frac{L_C^{(k)}(t^k)^{1/k}}{L_C(t)}=\frac{L_D^{(k)}(t^k)^{1/k}}{L_D(t)}
$$
and, raising to the $k$-th power, we get a polynomial equality
\begin{equation}\label{eqn-L_D,L_C}
L_C(t)^kL_D^{(k)}(t^k)=L_D(t)^kL_C^{(k)}(t^k)
\end{equation}
In particular, $L_C(t)^k$ divides $L_D(t)^kL_C^{(k)}(t^k)$. But
$$
L_C^{(k)}(t^k)=\prod_{i=1}^{2g(C)}(1-(\alpha_it)^k)=\prod_{i=1}^{2g(C)}\prod_{\zeta^k=1}(1-\zeta\alpha_i t)
$$
$$=L_C(t)\prod_{i=1}^{2g(C)}\prod_{\zeta^k=1,\zeta\neq 1}(1-\zeta\alpha_i t)
$$
so $L_C(t)^{k-1}$ divides $L_D(t)^k\prod_{i=1}^{2g(C)}\prod_{\zeta^k=1,\zeta\neq 1}(1-\zeta\alpha_i t)$. Since $L_C(t)$ and this last product are relatively prime by assumption, we conclude that $L_C(t)^{k-1}$ divides $L_D(t)^k$. Since $L_C(t)$ is square free, it must divide $L_D(t)$.

Having proved the divisibility of $L_D(t)$ by $L_C(t)$ we need to prove their quotient is of desired form.  Using~\eqref{eqn-L_D,L_C} and writing $L_D(t)=p(t)L_C(t)$, we have
$$
p(t)^k=L_D^{(k)}(t^k)/L_C^{(k)}(t^k).
$$
Now the claim follows from Lemma~\ref{radical-subrings} and there exists $q(t)$ in $\Z[t]$ so that $L_D(t)=q(t^k)L_C(t)$.
\end{proof}
We remark that the theorem becomes false when we replace the first hypothesis
``for every $m$ that is not divisible by $k$''
with
``for every $m$ with $\gcd(m,k)=1$.''
A counterexample is given by the curves (defined over $\F_3$)
$D:y^2+(x^2+x+1)y=x^5+x^4+x^2+x+1$ 
which has L-polynomial
$9t^4+3t^3-2t^2+t + 1$
and the curve
$C:y^2+(2x+1)y=x^3+2x^2+2$
which has L-polynomial
$3t^2 + t + 1$. The curve $C$ is an ordinary curve and it follows from Lemma~8 of \cite{AhShp} that $L_C^{(n)}(t)$ is an irreducible polynomial for every $n$ and hence it has distinct roots. Furthermore,
the curves $C$ and $D$ have the same number of rational points over $\F_{3^m}$
when $\gcd (m,6)=1$ but there is no divisibility of L-polynomials.
More generally, for a suitable $a$ a curve with L-polynomial 
$L_1: qt^2-at+1$ and a curve with L-polynomial $L_2: q^2t^4-aqt^3+(a^2-q)t^2-at+1$ have
the same number of rational points over $\F_{q^m}$
when $\gcd (m,6)=1$ but there is no divisibility of L-polynomials. The existence of curves of genus one with L-polynomial equal to $L_1$ for some $a$ and curves of genus two with L-polynomial equal to $L_2$ is guaranteed by the results on the classification of Weil polynomials of degree two and four~\cite{Deuring, Ruck}.

\subsection{A Converse}

We have the following theorem which is a partial converse of the theorem above.

\begin{thm}\label{converse}
Let $C$ and $D$ be two smooth projective curves over $\F_q$.
Assume that there exists a positive integer $k>1$ such that 
 $L_D(t)=q(t^k)L_C(t)$ for some polynomial $q(t)$ in $\Z[t]$. Then
 $\#C(\F_{q^m})=\#D(\F_{q^m})$ for every $m$ that is not divisible by $k$.
\end{thm}
\begin{proof}
Since
\[
  \frac{L_C(t)}{(1-t)(1-qt)}=Z_C(t)=exp \biggl( \sum_{m\geq 1} \#C(\F_{q^m})  \frac{t^m}{m} \biggr),
\]
it follows that  
\[
\log L_C(t)=\sum_{m\geq 1}( \#C(\F_{q^m})-1-q^m)  \frac{t^m}{m}
\]
in $\mathbb{Q}[[t]]$. Similarly, we have
\[
\log L_D(t)=\sum_{m\geq 1}( \#D(\F_{q^m})-1-q^m)  \frac{t^m}{m}.
\]
Now from $L_D(t)=q(t^k)L_C(t)$ we have
$\log L_D(t)=\log q(t^k)+\log L_C(t)$ in $\mathbb{Q}[[t]]$, and since $\log q(t^k)$ is a power series in $t^k$, the coefficients of $t^m$ in $\log L_C(t)$ and $\log L_D(t)$ are equal for $m$ not a multiple of $k$, and so are $\#C(F_{q^m})=\#D(F_{q^m})$ whenever $m$ is not a multiple of $k$.
\end{proof}

\section{Application to Curves}

We present a simple family of curves where our theorem 
(Theorem \ref{main}) applies, modulo a conjecture,
and yet the the Kani-Rosen and Kleiman-Serre theorems do not apply.
Our family is a subfamily of a family considered by Poonen~\cite{Poonen}.

We let $D_k$ be the hyperelliptic curve defined by the affine equation
\[
D_k: y^2+y=x^{2^k+1}+x^{-1}
\]
 over $\F_2$. 
We make the following conjecture, where the L-polynomials stated are over $\F_2$.

\begin{conjecture}\label{conjdk}
The L-polynomial of $D_k$ is divisible by the L-polynomial of $D_1$.
\end{conjecture}

In fact we will also make a more refined conjecture.

\begin{conjecture}\label{conjdk2}
Let $k=p_1^{a_1}\cdots p_m^{a_m}$ be the prime factorization of $k$,
where $p_1,\ldots,p_m$ are distinct primes.
Then 
\[
L_{D_k}(t)=q_1(t^{p_1})\cdots q_m (t^{p_m})L_{D_1}(t)
\]
for some polynomials $q_i(t)$ in $\Z[t]$. 
\end{conjecture}

We verified the conjectures for $k\leq 5$ using MAGMA \cite{MR1484478}.
The first five L-polynomials are

\bigskip

$
D_1: 4t^4 + 2t^3 + t + 1
$

$
D_2: (4t^4 + 2t^3 + t + 1)(2t^2 + 1)
$

$
D_3: (4t^4 + 2t^3 + t + 1)(8t^6 - 4t^3 + 1)
$

$
D_4: (4t^4 + 2t^3 + t + 1)(128t^{14} + 64t^{12} + 2t^2 + 1)
$

$
D_5: (4t^4 + 2t^3 + t + 1)(32768t^{30} + 4096t^{25} + 4t^5 + 1)
$

\bigskip

Multiplying the curve equation by $x^2$ and replacing $xy$ by $y$ shows that
the curve $D_k$ is birational to
\[
E_k: y^2+xy=x^{2^k+3}+x.
\]
Thus it follows that $D_k$ is of genus $2^{k-1}+1$. Now considering the degree 2 map 
$\psi: D_k\longrightarrow \mathbb{P}^1$ which maps the point $(x,y)$ to $x$, the ramification points are $P_0$ and $P_{\infty}$, where $P_0$ is the point with $x$-coordinate $0$ and $P_{\infty}$ is the point at infinity. 
It follows that the ramification divisor is $2 g_{D_k} P_0+ 2 P_{\infty}$. 
Using the Hurwitz genus formula this fact also shows that the genus of $D_k$ is $2^{k-1}+1$. Notice that the curve $D_k$ has one singularity at $P_{\infty}$.

\begin{lemma}\label{prankD}
The $2$-rank of $D_k$ is $1$.
\end{lemma}

\begin{proof}
By the Deuring-Shafarevitch formula, for an Artin-Schreier curve 
$y^p-y=f(x)$ the
$p$-rank is $m(p-1)$ where $m+1$ is the number of poles of $f(x)$.
Our curves $D_k$ have two poles, at 0 and $\infty$.
\end{proof}

Poonen~\cite{Poonen} studied the following family of curves
\[
L_g:y^2+y=x^{2g-1}+x^{-1}
\]
and showed that their automorphism group consists of the identity and the hyperelliptic involution. Notice that $L_g$ is of genus $g$, and the family of curves $L_g$ 
includes $D_k$. 
Computer experiments show that the analogue of Conjecture \ref{conjdk}
is false for the larger family of $L_g$ curves.
As we need the theorem of Poonen in the rest of this section we include its proof
 for the sake of completeness.
 
\begin{thm}\cite{Poonen}\label{Poonen}
The automorphism group of $L_g$ consists of the identity and the hyperelliptic involution.
\end{thm}

\begin{proof}
The ramification points of degree 2 map $\psi: L_g \longrightarrow \mathbb{P}^1$ which maps the point $(x,y)$ to $x$ are $P_0$ and $P_{\infty}$. It is well known that  the Weierstrass points of hyperelliptic curves are exactly the ramification points of the hyperelliptic involution and any automorphism fixes the set of Weierstrass points. It is also well known that the automorphism of the hyperelliptic curves are lifts of the automorphisms of $\mathbb{P}^1$. Thus it follows that automorphisms of $L_g$ are lifts of the maps taking $x$ to $\lambda x$ for some nonzero $\lambda$ in the algebraic closure of $\F_2$. But according to Artin-Schreier theory the following two 
\[
y_1^2+y_1=x^{2g-1}+x^{-1}
\]
and
\[
y_2^2+y_2=(\lambda x)^{2g-1}+(\lambda x)^{-1}
\]
are distinct.
\end{proof}

A consequence of Theorem \ref{Poonen} is that the Kani-Rosen method 
(Theorem \ref{kr}) will not
work if we want to use it to prove Conjecture \ref{conjdk}.

\bigskip

The second approach mentioned in the introduction that one might use to prove
Conjecture \ref{conjdk} is the Kleiman-Serre theorem, Theorem \ref{KleimanSerre}.
To apply this theorem one would have to show that there is a map
$D_k \longrightarrow D_1$ for any $k>1$.
If there were such a map, there would in particular be a map 
$D_2 \longrightarrow D_1$.
However, the following theorem shows that there is no  morphism from $D_2$ to $D_1$, and hence the Kleiman-Serre  theorem does not apply to our case.

\begin{thm}\label{nomap}
There is no non-constant morphism from the curve $D_{k+1}$ to $D_k$,
for any $k\geq 1$.
\end{thm}

\begin{proof}
Suppose there is a morphism
\[
\phi: D_{k+1}\longrightarrow D_k.
\]
Using Corollary~\ref{Separable}, we can assume that $\phi$ is a separable map.
Applying the Riemann-Hurwitz genus formula~\cite[Theorem 5.9]{Silverman} we have
\[
2g_{D_{k+1}}-2\ge \deg(\phi) (2g_{D_k}-2)+ \sum_{P\in D_{k+1}}( e(P)-1).
\]
Since $g_{D_{k+1}}=2^{k}+1$ and $g_{D_{k}}=2^{k-1}+1$, it follows that $\deg( \phi)=2$ and $\sum_{P\in D_{k+1}} (e(P)-1)=0$. So $D_{k+1}$ is an unramified double cover of $D_k$. Now $\phi$ is a separable map of degree 2. Thus
$D_{k+1}$ is a Galois cover of $D_k$. This would imply that there is an involution other than the hyperelliptic involution in the automorphism group of $D_{k+1}$ which contradicts Theorem~\ref{Poonen}.  Notice that here we need the separability of $\phi$ as our curves are defined over a field of characteristic two.
\end{proof}

\begin{cor}
There is no non-constant morphism from the curve $D_{2}$ to $D_1$.
\end{cor}

We have shown that the Kani-Rosen and Kleiman-Serre approaches will not
apply to prove Conjectures  \ref{conjdk} and \ref{conjdk2}.
%In the next sections we will discuss a conjecture that, if true,
%will prove Conjectures  \ref{conjdk} and \ref{conjdk2} via Theorem \ref{main}.

%It should be mentioned that n n-cover of a a hyperelliptic curve is Galois if and only if n=1,2,4.

\section{Motivation}\label{motive}

In this section we will give our motivation for studying the curves $D_k$,
which comes from a problem on certain exponential sums.

In \cite{JHK}, while studying the cross-correlation of m-sequences, Johansen, Helleseth and Kholosha considered the following exponential sums
\[
G_m^{(k)}=\sum_{x\in \F_{2^m}^*}(-1)^{Tr_m(x^{2^k+1}+x^{-1})}
\]
where $Tr_m(.)$ denotes the trace function from $\F_{2^m}$ to $\F_2$.
In particular, $G_m^{(1)}=\sum_{x\in \F_{2^m}^*}(-1)^{Tr_m(x^{3}+x^{-1})}$.
They made the following conjecture.

\begin{conjecture}\label{Helleseth-conj} 
 If $\gcd(k,m) = 1$ then
$G_m^{(k)}=G_m^{(1)}$.
\end{conjecture}

They also made a more general conjecture, that
$G_m^{(k)}= G_m^{(gcd(k,m))}$,
i.e., that $G_m^{(k)}$
 depends only on $\gcd(k,m)$. 

Now if $Tr_m(u^{3}+u^{-1})=0$ for $u\in\F_{2^m}$, then there
are two points on the following curve considered over $\F_{2^m}$ 
\[
D_k: y^2+y=x^{2^k+1}+x^{-1}
\]
with $x$-coordinate equal to $u$, and if $Tr_m(u^{3}+u^{-1})=1$ then there is no point on $D_k$ 
with $x$-coordinate equal to $u$. Thus we have
\[
 \# D_k (\F_{2^m})=2^m+1- G_m^{(k)}.
\]
Hence we may restate Conjecture~\ref{Helleseth-conj} as follows.

\begin{conjecture}
 When $\gcd(k,m)=1$ the curves $D_1$ and $D_k$
have the same number of rational points 
over $\F_{2^m}$.
\end{conjecture}

While investigating Conjecture~\ref{Helleseth-conj}, we looked at the zeta function of $D_k$
and we observed empirically using MAGMA \cite{MR1484478} that the L-polynomial of $D_1$ over $\F_{2}$ divides
the L-polynomial of $D_k$ over $\F_{2}$, and
we made Conjecture \ref{conjdk}.
In the next section we will prove that Conjecture~\ref{Helleseth-conj} is equivalent
to Conjecture \ref{conjdk2}.

\section{Equivalence of Conjectures}

In this section we consider the relationship between
Conjectures \ref{conjdk} and \ref{conjdk2} and Conjecture \ref{Helleseth-conj}.
%In fact we will prove some stronger theorems. 

\subsection{Conjecture \ref{conjdk2} implies Conjecture \ref{Helleseth-conj} }

%\begin{thm}\label{pp7to12}
%Let $k$ be a prime power.
%Then Conjecture \ref{conjdk2} implies Conjecture \ref{Helleseth-conj}.
%\end{thm}

%\begin{proof}
%This follows immediately  from Theorem \ref{converse}.
%\end{proof}

\begin{thm}\label{7to12}
For any $k\geq 1$, Conjecture \ref{conjdk2} implies Conjecture \ref{Helleseth-conj}.
\end{thm}

\begin{proof}
This follows from an argument similar to the proof of  Theorem \ref{converse}.
Let $k=p_1^{a_1}\cdots p_m^{a_m}$ be the prime factorization of $k$,
and suppose
\[
L_{D_k}(t)=q_1(t^{p_1})\cdots q_m (t^{p_m})L_{D_1}(t)
\]
for some polynomials $q_i(t)$ in $\Z[t]$. 
Taking the log on both sides leads to an equality of formal power series.
Each of the terms $\log q_i (t^{p_i})$ is a polynomial in $t^{p_i}$,
so for any $m$ that is relatively prime to $k$ the coefficient of
$t^m$ in all these terms is 0.
Therefore $L_{D_k}(t)$ and  $L_{D_1}(t)$ have the same coefficient of $t^m$
for any $m$ with $\gcd (m,k)=1$, and 
Conjecture \ref{Helleseth-conj} is true.
\end{proof}

\begin{cor}
Conjecture \ref{Helleseth-conj} is true for $k\leq 5$.
\end{cor}

\begin{proof}
Using Magma we computed the following L-polynomials:
\bigskip

$
D_1: 4t^4 + 2t^3 + t + 1
$

$
D_2: (4t^4 + 2t^3 + t + 1)(2t^2 + 1)
$

$
D_3: (4t^4 + 2t^3 + t + 1)(8t^6 - 4t^3 + 1)
$

$
D_4: (4t^4 + 2t^3 + t + 1)(128t^{14} + 64t^{12} + 2t^2 + 1)
$

$
D_5: (4t^4 + 2t^3 + t + 1)(32768t^{30} + 4096t^{25} + 4t^5 + 1).
$
\bigskip

Since Conjecture  \ref{conjdk2} is then clearly true for $k\leq 5$,
the result follows from Theorem \ref{7to12}.
\end{proof}

\subsection{Conjecture \ref{Helleseth-conj}  implies Conjecture \ref{conjdk}}

We shall prove that Conjecture \ref{Helleseth-conj} implies Conjecture \ref{conjdk} for general $k$, and Conjecture \ref{conjdk2} if $k$ has at most 2 prime powers. First we consider the case where $k$ is a prime power.

\begin{thm}\label{jacdiv1}
Let $k=p^a$ be a prime power. Then 
Conjecture \ref{Helleseth-conj} implies Conjecture \ref{conjdk2}.
%In other words, if $G_m^{(k)}=G_m^{(1)}$ for all $m$ with $(m,k)=1$
%then  $L_{D_k}(t)=q(t^p)L_{D_1}(t)$ for some polynomial $q(t)$ in $\Z[t]$. 
%In particular, the L-polynomial of $D_k $ is divisible by the L-polynomial of  $D_1$.
\end{thm}

\begin{proof}
This follows from Theorem \ref{main} (applied to $k=p$) with $C=D_1$ and $D=D_k$.
The first condition in Theorem \ref{main} holds by the hypotheses, 
since being relatively prime to $p^a$ is equivalent to not being divisible by $p$.
To show the second condition in Theorem \ref{main} it suffices to show that
the L-polynomial of $D_1(\F_{2^r})$ for each $r$ is irreducible.
Since $D_1$ has genus 2, its L-polynomial has degree 4 and the
L-polynomial for $\F_2$ is easily calculated to be $4t^4+2t^3+t+1$
which is irreducible. This implies that the Jacobian of $D_1(\F_{2})$ is simple.
By \cite{maisner-nart} any simple abelian surface of 2-rank 1 is absolutely simple, so the Jacobian is absolutely simple. 
By Lemma 2.12 of  \cite{maisner-nart}  the L-polynomial of $D_1(\F_{2^r})$ for each $r$ is therefore irreducible.
\end{proof}

For general $k$ we will first prove that
Conjecture~\ref{Helleseth-conj} implies Conjecture \ref{conjdk},
after a small lemma.

\begin{lemma}\label{maninapp}
Let $P\in 1+t{\mathbb Z}[t]$ be the L-polynomial of a curve with $p$-rank $0$.
Let $Q\in 1+t{\mathbb Z}[t]$ be the L-polynomial of a curve with $p$-rank $>0$.
Then $Q(t)$ does not divide $P(t)$.
\end{lemma}

\begin{proof}
A theorem of Manin states that the $p$-rank is equal to the degree
of the mod $p$ reduction of the L-polynomial.
Therefore the mod $p$ reduction of $P(t)$, denoted $\overline{P(t)}$,  is 1.

Let $d(t)=\gcd (P(t),Q(t))$. Then  $\overline{d(t)}=1$, and $\overline{Q(t)}\not= 1$,
so $d(t)\not= Q(t)$.
\end{proof}

\begin{thm}\label{12to6}
Conjecture~\ref{Helleseth-conj} implies Conjecture \ref{conjdk}.
\end{thm}

\begin{proof}
We proceed by induction on the number of prime divisors of $k$. If $k=p^a$ for some prime $p$, then this follows from Theorem \ref{jacdiv1}.
Let  $k=p_1^{a_1}\cdots p_m^{a_m}$ where $p_1,\ldots,p_m$ are distinct primes, and let $k'=k/p_m^{a_m}$. 

Conjecture~\ref{Helleseth-conj} implies that $\#D_k(\F_{2^r})=\#D_{k'}(\F_{2^r})$ for every $r$ which is not divisible by $p_m$. An argument similar to the proof of Theorem~\ref{main} shows that 
\[
L_{D_k}(t)^{p_m}L_{D_{k'}}^{(p_m)}(t^{p_m})=L_{D_{k'}}(t)^{p_m}L_{D_k}^{(p_m)}(t^{p_m}). 
\]

By induction hypothesis, $L_{D_{k'}}(t)=L_{D_1}(t)\cdot P(t)$ for some 
$P\in 1+t{\mathbb Z}[t]$. Note that, since both $L_{D_{k'}}$ and $L_{D_1}$ have $p$-rank $1$, $P$ must have $p$-rank $0$. 
Then we have
$$
L_{D_k}(t)^{p_m}L_{D_{1}}^{(p_m)}(t^{p_m})P^{(p_m)}(t^{p_m})=L_{D_{1}}(t)^{p_m}P(t)^{p_m}L_{D_k}^{(p_m)}(t^{p_m}).
$$
Since the $p_m$-th powers of the roots of $L_{D_1}$ are distinct, we can deduce as in the proof of Theorem~\ref{main} that $L_{D_{1}}(t)^{p_m-1}$ divides $L_{D_k}(t)^{p_m}P^{(p_m)}(t^{p_m})$. However $L_{D_{1}}(t)$ cannot divide $P(t^{p_m})$, as the latter has $p$-rank $0$, by Lemma \ref{maninapp}. 
Therefore, since $L_{D_{1}}(t)$ is irreducible it must divide $L_{D_k}(t)$.
\end{proof}

Before extending the argument to show that Conjecture~\ref{Helleseth-conj} implies Conjecture \ref{conjdk2} when $k$ has at most 2 prime factors, we prove two lemmas.

\begin{lemma}\label{radical-subring-2}
Suppose that $f_1(x),f_2(x),g_1(x)$ and $g_2(x)$ are polynomials in $\Z[x]$ where $f_i(0)\neq 0$ for $i=1,2$  and $(f_1(x)/f_2(x))^n=g_1(x^k)/g_2(x^k)$ for positive integers $k$ and $n$. Then there are polynomial $h_1(x)$ and $h_2(x)$ in $\Z[x]$ so that $f_1(x)/f_2(x)=h_1(x^k)/h_2(x^k)$. 
\end{lemma}

\begin{proof}
Without loss of generality we may assume that $f_1$ and $f_2$ are coprime. Let $\zeta_k$ be a primitive $k$-th root of unity. Then 
\begin{eqnarray*}
(f_1(\zeta_kx)/f_1(\zeta_kx))^n&=&g_1((\zeta_kx)^k)/g_2((\zeta_kx)^k)\\
&=&g_1(x^k)/g_2(x^k)\\
&=&(f_1(x)/f_2(x))^n.
\end{eqnarray*}
Thus 
\[ \frac{f_1(x)}{f_2(x)}=\zeta_n \frac{f_1(\zeta_kx)}{f_2(\zeta_kx)}
\]
 for some $n$-th root of unity $\zeta_n$. If we let $x=0$, from the fact that $f_i(0)\neq 0$ for $i=1,2$ we derive that $\zeta_n=1$, and hence
\[ \frac{f_1(x)}{f_2(x)}= \frac{f_1(\zeta_kx)}{f_2(\zeta_kx)}
\]
which is equivalent to  
$$
f_1(x)f_2(\zeta_kx)=f_2(x) f_1(\zeta_kx).
$$
Now we know that $f_1$ and $f_2$ have no common root, so from the equation above it follows that $a$ is a root of $f_1$ with multiplicity $l$ if and only if $\zeta_k a$ is a root of $f_1$ with multiplicity $l$. Thus $f_1(x)=h_1(x^k)$ for some $h_1$. Similarly,  $f_2(x)=h_2(x^k)$ for some $h_2$.

\end{proof}

\begin{lemma}\label{rattopoly}
 Let $f\in 1+T\cdot{\mathbb Z}[T]$ and let $p,q$ be relatively prime integers such that there exist polynomials $g_1,g_2,h_1,h_2\in 1+T\cdot{\mathbb Z}[T]$ with
$$
f(t)=\frac{g_1(t^p)g_2(t^q)}{h_1(t^p)h_2(t^q)}.
$$
Then there exist polynomials $f_1,f_2\in 1+T\cdot{\mathbb Z}[T]$ such that $f(t)=f_1(t^p)f_2(t^q)$. 
\end{lemma}

\begin{proof}
We may assume without loss of generality that $g_1,h_1$ are relatively prime (if they have a common factor, then they have a commmon factor of the form $P(t^p)$). Similarly with $g_2,h_2$.

We proceed by induction on $n=\deg(f)+\deg(h_1)+\deg(h_2)$, the case $n=0$ being obvious. Let $\alpha$ be a (reciprocal) root of $f(t)$. Then $\alpha$ is a root of either $g_1(t^p)$ or of $g_2(t^q)$. Without loss of generality, let us assume that it is a root of the former. Then $\zeta_p^i\alpha$ is also a root for every $i=0,\ldots,p-1$, where $\zeta_p$ is a primitive $p$-th root of unity. We now distinguish two cases:

If $\zeta_p^i\alpha$ is a root of $f(t)$ for every $i=0,\ldots,p-1$, then $f(t)$ and $g_1(t^p)$ are divisible by $\prod_i(1-\zeta_p^i\alpha T)=1-\alpha^p T^p$. Let $a(t^p)\in 1+T^p{\mathbb Z}[T^p]$ be the product of all Galois conjugates of $1-\alpha^pt^p$. Then $f(t)$ and $g_1(t^p)$ are divisible by $a(t^p)$, and we have
$$
\frac{f(t)}{a(t^p)}=\frac{\frac{g_1(t^p)}{a(t^p)}g_2(t^q)}{h_1(t^p)h_2(t^q)}.
$$
By induction hypothesis there exist $\hat f_1,\hat f_2\in 1+T\cdot{\mathbb Z}[T]$ such that $\frac{f(t)}{a(t^p)}=\hat f_1(t^p)\hat f_2(t^q)$, and we take $f_1(t)=a(t)\hat f_1(t)$, $f_2(t)=\hat f_2(t)$.

Now suppose that there exists some $i_0$ such that $\zeta_p^{i_0}\alpha$ is not a root of $f(t)$. Then from $f(t)h_1(t^p)h_2(t^q)=g_1(t^p)g_2(t^q)$ we get that $\zeta_p^{i_0}\alpha$ is a root of $h_2(t^q)$, and then so is $\zeta_q^j\zeta_p^{i_0}\alpha$ for every $j=0,\ldots,q-1$, where $\zeta_q$ is a primitive $q$-th root of unity. Since $g_2$ and $h_2$ are relatively prime, $\zeta_q^j\zeta_p^{i_0}\alpha$ is a root of $g_1(t^p)$ for every $j=0,\ldots,q-1$. Then $\zeta_q^j\zeta_p^i\alpha$ is a root of $g_1(t^p)$ for every $i,j$, so $g_1(t^p)$ is divisible by $1-\alpha^{pq}t^{pq}$. Let $b(t^{pq})$ be the product of all Galois conjugates of $1-\alpha^{pq}t^{pq}$, then $g_1(t^p)$ is divisible by $b(t^{pq})$, let $g_1(t^p)=g_1'(t^p)b(t^{pq})$. We have
$$
f(t)h_1(t^p)h_2(t^q)=g_1'(t^p)(b(t^{pq})g_2(t^q)).
$$

If $b(t^{pq})$ and $h_2(t^q)$ are not relatively prime, then they have a common factor of the form $p(t^q)$. Dividing both of them by that factor we get $f(t)h_1(t^p)h'_2(t^q)=g_1'(t^p)g'_2(t^q)$, and we conclude by induction hypothesis since $\deg(h'_2)<\deg(h_2)$. Otherwise, $b(t^{pq})$ must divide $f(t)$. Writing $f(t)=f'(t)b(t^{pq})$ we get
$$
f'(t)h_1(t^p)h_2(t^q)=g_1'(t^p)g_2(t^q)
$$
and again we conclude by induction.
\end{proof}

Finally we prove the main result of this section.

\begin{thm}
If $k$ has at most 2 prime factors, Conjecture~\ref{Helleseth-conj} implies Conjecture \ref{conjdk2}.
\end{thm}

\begin{proof}
The prime power case is Theorem \ref{jacdiv1}, so suppose that $k=p_1^{a_1}p_2^{a_2}$.

Following the proof of Theorem \ref{12to6}, 
and assuming that $$L_{D_{k'}}(t)=P(t)L_{D_1}(t)$$ and  $$L_{D_{k}}(t)=Q(t) L_{D_1}(t)$$ 
from the main equation in Theorem \ref{12to6} we get that
$$
Q(t)^{p_2} P(t)^{(p_2)}(t^{p_2})=P(t)^{p_2} Q(t)^{(p_2)}(t^{p_2}).
$$ 
This is equivalent to 
$$
\biggl(\frac{Q(t)}{P(t)}\biggl)^{p_2}=\frac{Q^{(p_2)}(t^{p_2})}{P^{(p_2)}(t^{p_2})}.
$$
Applying Lemma \ref{radical-subrings}, we have that
$$
\frac{Q(t)}{P(t)}=\frac{h_1(t^{p_2})}{h_2(t^{p_2})}.
$$
From this and Theorem  \ref{jacdiv1} we have
 $$L_{D_{k}}(t)=Q(t) L_{D_1}(t)=\frac{h_1(t^{p_2})}{h_2(t^{p_2})}P(t)L_{D_1}(t)=\frac{h_1(t^{p_2})}{h_2(t^{p_2})}q_1(t^{p_1})L_{D_1}(t)$$
 for some $q_1(t)\in\Z[t]$.
 By Lemma \ref{rattopoly} the proof is complete.
\end{proof}
When $k$ has more than 2 prime factors,  it is not hard to see that the proofs of the theorem above and Theorem~\ref{12to6} along with an induction can be used to prove that Conjecture~\ref{Helleseth-conj} implies a weaker version of Conjecture~\ref{conjdk2} where the polynomials in Conjecture~\ref{conjdk2} are replaced with ratios of polynomials. More precisely, we have the following theorem. Notice that the proof of the converse claim in the following is similar to the proof of Theorem~\ref{converse}.
\begin{thm}
Let $k=p_1^{a_1}\cdots p_m^{a_m}$ be the prime factorization of $k$,
where $p_1,\ldots,p_m$ are distinct primes. Then Conjecture~\ref{Helleseth-conj} implies that
\[
L_{D_k}(t)=\frac{q_{m,1}(t^{p_m})}{q_{m,2}(t^{p_m})}\frac{q_{m-1,1}(t^{p_{m-1}})}{q_{m-1,2}(t^{p_{m-1}})}\cdots\frac{q_{2,1}(t^{p_{2}})}{q_{2,2}(t^{p_{2}})} q_1 (t^{p_1})L_{D_1}(t)
\]
for some polynomials $q_{i,1}(t),q_{i,2}(t)$ and $q_1(t)$ in $\Z[t]$. Conversely, if $ L_{D_k}(t)$ is in the above form, then Conjecture~\ref{Helleseth-conj} is true. 
\end{thm} 
 
\section{Acknowledgements}
We would like to thank Robert Granger, Igor Shparlinski and Alexey Zaytsev for many insightful discussions during the preparation of this paper.

 %%%%%%%%%%%%%%%%%%%%%%%%%%%%%%%%%%%%%%%%%%%%%%%%%%%%%%%%%%%


\begin{thebibliography}{99}

\bibitem{AhShp}
Omran Ahmadi and Igor~E. Shparlinski.
\newblock On the distribution of the number of points on algebraic curves in
  extensions of finite fields.
\newblock {\em Math. Res. Lett.}, 17(4):689--699, 2010.

\bibitem{AubryPerret}
Y.~Aubry and M.~Perret.
\newblock Divisibility of zeta functions of curves in a covering.
\newblock {\em Archiv. der Math.}, 82:205--213, 2004.

\bibitem{journals/moc/BauerTW05}
Mark Bauer, Edlyn Teske, and Annegret Weng.
\newblock Point counting on picard curves in large characteristic.
\newblock {\em Math. Comput.}, 74(252):1983--2005, 2005.

\bibitem{MR1484478}
W.~Bosma, J.~Cannon, and C.~Playoust.
\newblock The {M}agma algebra system. {I}. {T}he user language.
\newblock {\em J. Symbolic Comput.}, 24(3-4):235--265, 1997.
\newblock Computational algebra and number theory (London, 1993).

\bibitem{Deuring}
Max Deuring.
\newblock Die {T}ypen der {M}ultiplikatorenringe elliptischer
  {F}unktionenk\"orper.
\newblock {\em Abh. Math. Sem. Hansischen Univ.}, 14:197--272, 1941.

\bibitem{JHK}
Aina Johansen, Tor Helleseth, and Alexander Kholosha.
\newblock Further results on {$m$}-sequences with five-valued cross
  correlation.
\newblock {\em IEEE Trans. Inform. Theory}, 55(12):5792--5802, 2009.

\bibitem{KaniRosen}
E.~Kani and M.~Rosen.
\newblock Idempotent relations and factors of {J}acobians.
\newblock {\em Math. Ann.}, 284:307--327, 1989.

\bibitem{Kleiman}
S.~L. Kleiman.
\newblock Algebraic cycles and the {W}eil conjectures.
\newblock In {\em Dix espos\'es sur la cohomologie des sch\'emas}, pages
  359--386. North-Holland, Amsterdam, 1968.

\bibitem{maisner-nart}
D.~Maisner and E.~Nart.
\newblock Abelian surfaces over finite fields as {J}acobians.
\newblock {\em Experiment. Math.}, 11(3):321--337, 2002.
\newblock With an appendix by Everett W. Howe.

\bibitem{MR2648557}
Gary McGuire and Alexey Zaytsev.
\newblock On the zeta functions of an optimal tower of function fields over
  {$\mathbb{F}_4$}.
\newblock In {\em Finite fields: theory and applications}, volume 518 of {\em
  Contemp. Math.}, pages 327--338. Amer. Math. Soc., Providence, RI, 2010.

\bibitem{Paulhus_decomposingjacobians}
Jennifer Paulhus.
\newblock Decomposing jacobians of curves with extra automorphisms.
\newblock {\em Acta Arithmetica}, 132:231--244, 2008.

\bibitem{Poonen}
Bjorn Poonen.
\newblock Varieties without extra automorphisms. {II}. {H}yperelliptic curves.
\newblock {\em Math. Res. Lett.}, 7(1):77--82, 2000.

\bibitem{rv}
P.~{Roquette}.
\newblock {The Riemann hypothesis in characteristic p, its origin and
  development Part 1. The formation of the zeta-functions of Artin and of F.K.
  Schmidt}.
\newblock {\em http://www.rzuser.uni-heidelberg.de/$\sim$ ci3/rv.pdf}, July
  2003.

\bibitem{Ruck}
Hans-Georg R{\"u}ck.
\newblock Abelian surfaces and {J}acobian varieties over finite fields.
\newblock {\em Compositio Math.}, 76(3):351--366, 1990.

\bibitem{Silverman}
Joseph~H. Silverman.
\newblock {\em The arithmetic of elliptic curves}, volume 106 of {\em Graduate
  Texts in Mathematics}.
\newblock Springer, Dordrecht, second edition, 2009.

\bibitem{DBLP:books/daglib/0084861}
Henning Stichtenoth.
\newblock {\em Algebraic function fields and codes}.
\newblock Universitext. Springer, 1993.

\bibitem{waterhouse}
W.~C. Waterhouse.
\newblock Abelian varieties over finite fields.
\newblock {\em Ann. Sci. \'Ecole Norm. Sup. (4)}, 2:521--560, 1969.

\end{thebibliography}
\end{document}